\documentclass[english,11pt,a4paper]{smfart}

\usepackage[english]{babel}
\usepackage{amsmath}
\usepackage{amsfonts}
\usepackage{amssymb}
\usepackage{amsthm}
\usepackage{a4wide}
\usepackage{mathrsfs}
\usepackage[T1]{fontenc}
\usepackage{url}

\newtheorem{theorem}{Theorem}
\newtheorem{dfn}{Definition}
\newtheorem{lem}{Lemma}

\newtheorem{prop}{Proposition}

\def\bn{\mathbf{n}}

\def\C{\mathbb{C}}

\def\N{\mathbb{N}}

\def\Z{\mathbb{Z}}
\def\di{\displaystyle}
\def\ni{\noindent}

\begin{document}
\title{Lie algebras and geometric complexity of an isochronous center condition}
\author{Jacky Cresson}
\author{Jordy Palafox}

\begin{abstract}
Using the mould formalism introduced by Jean Ecalle, we define and study the geometric complexity of an isochronous center condition. The role played by several Lie ideals is discussed coming from the interplay between the universal mould of the correction and the different Lie algebras generated by the comoulds. This strategy enters in the general program proposed by J. Ecalle and D. Schlomiuk in \cite{es} to study  the size and splitting of some Lie ideals for the linearisability problem.
\end{abstract}

\maketitle

\setcounter{tocdepth}{3}
\tableofcontents

\section{Introduction}

In this article, we are interested in {\bf characterizing isochronous center} of {\bf polynomial real vector fields} which can be written using complex coordinates as (see \cite{f}):
\begin{equation}
\label{form}
X=\xi(x\partial_x -y\partial_y )+P(x,y)\partial_x +Q(y,x)\partial_y,
\end{equation}
where $\xi \in \C$ such that $\xi^2=-1$, $\partial_x=\frac{\partial}{\partial x}$, $\partial_{y}=\frac{\partial}{\partial y}$, $P(x,y)$ is a polynomial of a maximal degree $d$ given by $P(x,y)=\underset{2\leq i+j\leq d}{\sum}p_{i,j}x^iy^j$, with $\bar{y}=x$ and the coefficients $p_{i,j} \in \C$ satisfy $q_{i,j}=\overline{p_{j,i}}$.

\begin{dfn}
A vector field is said to be isochronous if all its orbits in a neighbourhood of an equilibrium point are periodic with the same period.
\end{dfn}

We have a characterization of isochronicity (see \cite{chava}) : 

\begin{theorem}
Such a vector field $X$ is isochronous if and only if $X$ is linearisable by an analytic change of coordinates.
\end{theorem}

In this article, we use two classes of objects to study linearisation: {\bf prenormal forms} and the {\bf correction} introduced by J. Ecalle and B. Vallet in \cite{ev2}. The main characteristic of these objects is that they are {\bf computable algorithmically} and thus provide {\bf explicit criterion of linearisability} and moreover that they possess a {\bf rich algebraic structure} coming from the use of the {\bf mould formalism}. The linearisability is equivalent to the fact that any prenormal form of a given vector field is reduced to the linear part or its correction is trivial (see \cite{ev2}).\\

More precisely, let us consider a vector field in a {\it prepared form}:
\begin{equation}
X=X_{lin}+\underset{n\in A(X)}{\sum}B_{n},
\end{equation}
where $X_{lin}=\xi(x\partial_x-y\partial_{y})$, $x=\bar{y}$, the $B_n$ are homogeneous differential operators of degree $n=(n_1,n_2)\in \Z^2$, with $n_1 \geq 0$, $n_2\geq 0$ and all most one of the $n_i$ is equal to $-1$ i.e $B_n(x^m)=c_{n,m}x^{n+m}$, $c_{n,m} \in \C$, $m\in \N^2$, $x=(x_1,x_2)$ and $x^{m}=x_1^{m_1}x_2^{m_2}$. The set of degree associated to $X$ is denoted by $A(X)$. \\

Considering a polynomial vector field of the form (\ref{form}), the previous decomposition leads to three types of differential operators:
\begin{equation}
\left .
\begin{array}{lll}
B_{(i-1,k-i)} & = & x^{i-1}y^{k-i}(p_{i,k-i}x\partial_x+\bar{p}_{k-i+1,i-1} y\partial_y) ,\\
B_{(-1,k)} & = & p_{0,k}y^k\partial_x ,\\
B_{(k,-1)} & = & \bar{p}_{0,k}x^k\partial_{y} ,
\end{array}
\right .
\end{equation}
where $2\leq k \leq d$ and $1 \leq i \leq k$.\\

A {\bf prenormal} form associated to $X$ is an expression of the form :
\begin{align*}
X=X_{lin}+\underset{\textbf{n}\in A^*(X)}{\sum}M^{\textbf{n}}B_{\textbf{n}},
\end{align*}
where $\textbf{n}$ is a word obtained by concatenation of letters in $A(X)$, $A^*(X)$ is the set of words constructed on $A(X)$, and for $\textbf{n}=n_1\cdot n_2 \cdot ...\cdot n_r \in A^* (X)$, $B_{\textbf{n}}=B_{n_1}\circ B_{n_2}\circ \cdots \circ B_{n_r}$ denotes the usual composition of differential operators and $M^{\mathbf{n}} \in \C$. Following J. Ecalle \cite{ec1}, we denote by $M^{\bullet}$ and we call {\bf mould} the map from $A^* (X)$ to  $\C$ defined $M^{\bullet} [\mathbf{n} ] =M^{\mathbf{n}}$. The corresponding map for the differential operators $B_{\mathbf{n}}$ is called a {\bf comould}.

\begin{dfn}
Let $\omega : (A^* (X) , .) \rightarrow (\Z , +)$ be the morphism defined for any letter $n\in A(X)$ by $\omega (n)=\langle n, \lambda\rangle$, with $\lambda=(\xi,-\xi)$. Then we have for any word $\textbf{n}=n_1\cdot ... \cdot n_r$ that  $\omega (\mathbf{n} )=\underset{j=1}{\overset{r}{\sum}} \omega (n_j )$. For $\mathbf{n} \in A^* (X)$, the quantity $\omega (\mathbf{n} )$ is called the weight of the word $\mathbf{n}$. A word $\mathbf{n} \in A^*(X)$ is said to be resonant if $\omega (\mathbf{n} )=0$.
\end{dfn}

A fundamental role is played by resonant words as they are in one to one correspondence with monomials in the vector field $X$ which forming an obstruction to linearisability. Precisely, we have :
 
\begin{prop}
The mould $M^{\bullet}$ of a prenormal form satisfies $M^{\textbf{n}}=0$ for non-resonant word $\textbf{n}$.
\end{prop}

In the same way, the {\it correction} of $X$ is given by $Carr(X)=\underset{\bullet}{\sum}Carr^{\bullet}B_{\bullet}$ where the mould $Carr^{\bullet}$ is defined algorithmically (see \cite{ev2}) and satisfies $Carr^{\varnothing} =0$ and $Carr^{\mathbf{n}} =0$ for all non-resonant word $\textbf{n} \in A^* (X)$. \\

Using these two objects, the linearisation is equivalent to the fact that:
\begin{equation}
\underset{\bullet}{\sum}M^{\bullet}B_{\bullet} =0 ,
\end{equation}
where the mould $M^{\bullet}$ is either the mould $Carr^{\bullet}$ of the correction or a mould $Pran^{\bullet}$ of a given prenormal form. 
{\it Isochronicity} is then equivalent to find the conditions on the coefficients of a polynomial $P$ such that such a formal series vanishes. \\

The interest of the mould formalism with respect to others approaches is that it separates the part depending on the coefficients of $P$ (the comould contribution) from the universal one depending only on the alphabet $A(X)$ generated by the vector field (the mould part). Moreover, as already pointed out by J. Ecalle and D. Schlomiuk (see \cite{es},$\S$.10,p.1474), the main difficulty in order to solve the isochronicity problem is that the "ideal $\mathbb{I}$ generated by (finitely many) Taylor coefficients of $X$ is unwieldy, unstructured and lacking in truly canonical bases." They propose to replace the commutative ideal $\mathbb{I}$ by some Lie ideals which arise naturally in the mould formalism approach (See \cite{es},$\S$.10,p.1475-1475 for more details).\\

The aim of this paper is to {\bf interpret some classical isochronous center conditions in term of Lie ideals} following the general {\bf program proposed by J. Ecalle and D. Schlomiuk}. We are leaded to define a notion of {\bf geometric complexity} for an {\bf isochronous center condition}.\\

The plan of the paper is as follows: by studying first the isochronous center conditions for quadratic polynomials, we observe that the {\bf uniform} and {\bf holomorph isochronous centers} are related to the {\bf nilpotence} and {\bf triviality} of the Lie algebra generated by the family of differential operators $B_n$, $n\in A(X)$ and the {\bf resonant} set respectively. These conditions do not depend on the value of the underlying mould and in some way are the {\bf simplest} one. We then generalize these results for polynomial vector fields with an arbitrary large degree improving on an unpublished paper of B. Schuman \cite{schu}. We then discuss the notion of geometric complexity for an isochronous center condition and state a conjecture.

\section{Quadratic isochronous center}

We consider a real quadratic planar vector field $X$ in its complex representation on $\C^2$ defined by:
\begin{equation}
\left\lbrace
\begin{array}{lll}
\dot{x} & = & \xi x+p_{2,0}x^2+p_{1,1}xy+p_{0,2}y^2 ,\\
\dot{y} & = & -\xi y +\bar{p}_{0,2}x^2+\bar{p}_{1,1}xy+\bar{p}_{2,0}y^2,
\end{array}
\right.
\end{equation}
where $y=\bar{x}$, $p_{i,j}\in \C$, $i+j=2$, $i=0,1,2$. \\

The following results can be founded in \cite{fp,zola}:

\begin{theorem}
A real quadratic planar vector field $X$ is an isochronous center if and only if at least one of the following conditions is satisfied:
\begin{align*}
& i) \ p_{1,1}=p_{0,2}=0 \ \ (\mbox{\rm Holomorphic isochronous center}), \\
& ii) \ p_{2,0}=\bar{p}_{1,1}, \  p_{0,2}=0  \ \ (\mbox{\rm Uniform isochronous center}), \\
& iii) \ p_{2,0}=\frac{5}{2}\bar{p}_{1,1}, \  \vert p_{1,1}\vert^2=\frac{4}{9}\vert p_{0,2}\vert^2, \\
& iv) \ p_{2,0}=\frac{7}{6}\bar{p}_{1,1}, \vert p_{1,1}\vert^2=4\vert p_{0,2}\vert^2 . \\
\end{align*}
\end{theorem}

Our approach suggests to look for these conditions by studying first their consequences on the Lie algebra generated by the $B_n$, $n\in A(X)$. 

\begin{lem}
\label{fond2}
A quadratic vector field satisfying
\begin{equation}
p_{2,0}=\bar{p}_{1,1}, \  p_{0,2}=0 ,
\end{equation}
or
\begin{equation}
p_{1,1}=p_{0,2}=0 ,
\end{equation}
is such that the Lie algebra $\mathfrak{b}$ generated by $\mathbf{B}(X) =\{ B_n \} _{n\in A(X)}$ is nilpotent of order 1, i.e. for all $n,\, m \in A (X)$ we have 
\begin{equation}
[B_n ,B_m ] =0 .
\end{equation}
\end{lem}

Using this Lemma and the structure of the series, we recover easily that these conditions correspond to an isochronous center. Indeed, we have the following general observation :

\begin{lem}
\label{fond3}
If the Lie algebra $\mathfrak{b}$ is nilpotent of order $1$ then any prenormal form associated to a mould $Pran^{\bullet}$ is reduced to 
\begin{equation}
Pran (X) = X_{lin} + \di\sum_{n\in A(X),\ \omega (n)=0} Pran^n B_n ,
\end{equation}
and the correction is given by
\begin{equation}
Carr (X) =\di\sum_{n\in A(X),\ \omega (n)=0} Carr^n B_n .
\end{equation}
\end{lem}

\begin{proof}
By the classical projection Theorem (see \cite{jps} Theorem 8.1 p.28), we have 
\begin{equation}
\di\sum_{\bullet} M^{\bullet} B_{\bullet} =\di\sum_{r\geq 1} \di\frac{1}{r} \di\sum_{\bn \in A^* (X)} 
M^{\bn} [B_{\bn} ] ,
\end{equation}
where $[B_n ] =B_n$ for $n\in A(X)$ and $[B_{\bn} ]=[B_{n_r} ,[B_{n_{r-1}} ,\dots ,[B_{n_2} ,B_{n_1} ] \dots ]]$ for $\bn =n_1 \dots n_r \in A^* (X)$, $r\geq 2$. Indeed the constant term of the series must be zero as this formal series corresponds to a vector field.\\

As a consequence, under these conditions we obtain
\begin{equation}
\label{projprenor}
\di\sum_{\bullet} M^{\bullet} B_{\bullet} =\di\sum_{r\geq 1} \di\frac{1}{r} \di\sum_{n \in A (X)} 
M^{n} B_n ,
\end{equation}
The mould $M^{\bullet}$ is equal to zero on non-resonant words so that the right side is reduce to a sum over resonant letters. 
\end{proof}

As already noted, the value and nature of the moulds are not important. Only the Lie algebraic structure of $\mathfrak{b}$ is taken into account.\\

The proof that conditions i) and ii) lead to isochronous center is then easily deduces. Indeed, quadratic vector fields do not produce resonant letters. As a consequence, Lemma \ref{fond3} implies that $Pran(X)=X_{lin}$ and $Carr(X)=0$.\\

The proof of Lemma \ref{fond2} goes as follows : First, as $p_{0,2}=0$ in the two cases, we have $B_{(2,-1)}=B_{(-1,2)}=0$. A simple computation gives 
\begin{equation}
[B_{(1,0)} ,B_{(0,1)} ]= x^2 y \left [ p_{1,1} \left ( \bar{p}_{1,1} -p_{2,0} \right ) \partial_x +
\bar{p}_{1,1} \left ( - p_{1,1} +\bar{p}_{2,0} \right ) \partial_y \right ] .
\end{equation}
As a consequence, a condition for $\mathfrak{b}$ to be nilpotent of order $1$ is 
\begin{equation}
p_{1,1} \left ( \bar{p}_{1,1} -p_{2,0} \right ) =0 .
\end{equation}

Then if $p_{1,1} =0$ or $\bar{p}_{1,1} =p_{2,0}$ the Lie algebra generated by the vector fields in $\mathbf{B}(X)$ is nilpotent of order one.\\

We generalize these conditions for a homogeneous polynomial perturbation of degree $d\geq 2$ in the following.

\section{Lie algebras generated by polynomial vector fields}

The previous Section indicates that the Lie algebra $\mathfrak{b}$ generated by the set $\mathbf{B}(X)$ of comoulds associated to $X$ plays a central role in the understanding of some center conditions. In this Section, we derive some useful results which will be used in our study of uniform and holomorph center conditions.

\subsection{Descending central series and Nilpotent Lie algebra}
In this section, we give some reminders about Lie algebras. We refer to \cite{reut} and \cite{jps} for more details.

\begin{dfn}
Let $\mathfrak{g}$ be a Lie algebra, we define its descending central series by :
\begin{align*}
\mathcal{C}^1(\mathfrak{g})=\mathfrak{g}, \\
\mathcal{C}^2(\mathfrak{g})=[\mathfrak{g},\mathfrak{g}], \\
\mathcal{C}^{n+1}(\mathfrak{g})=[\mathfrak{g},\mathcal{C}^n(\mathfrak{g})] \ \text{for} \ i\geq 2.
\end{align*}
\end{dfn}

Due to Lemma \ref{fond3}, we are interested in Nilpotent Lie algebra which are defined as follows :

\begin{dfn}
A Lie algebra $\mathfrak{g}$ is nilpotent if there exists an integer $i$ such that $\mathcal{C}^i(\mathfrak{g})$.
\end{dfn}

\subsection{Preliminaries}
We begin with some computations on Lie brackets:

\begin{lem}
Let $X$ be a polynomial vector field of degree $d \geq 2$ and $\mathbf{B}(X)$ its associated set of homogeneous differential operators. We have for $n=2,\dots ,d$ and $i=1,\dots ,n$:
\begin{equation}
\left . 
\begin{array}{lll}
[B_{(n,-1)} , B_{(-1,n)}] & = & n\vert p_{0,n}\vert ^2 (xy)^{n-1}
\left ( x\partial_x -y\partial_y \right ) ,\\

[ B_{(i-1,n-i)} , B_{(n-i,i-1)} ] & = & \left [ (n-i) p_{i,n-i} U_i -(i-1) p_{n-i+1,i-1} \bar{U}_i \right ] (xy)^{n-1}\partial_x -\\

 & & \left [ (n-i) \bar{p}_{i,n-i} \overline{U}_i +(i-1) \bar{p}_{n-i+1 ,i-1} U_i \right ] (xy)^{n-1}\partial_y ,
\end{array}
\right .
\end{equation}
where $U_i = p_{n-i+1,i-1} -\bar{p}_{i,n-i}$.
\end{lem}

The previous Lemma gives a special role to the quantities $U_i$, $i=1,\dots ,n$ and also to the coefficients $p_{i,n-i}$. Using this remark, we are able to derive a class of explicit conditions for which the associated Lie algebra $\mathfrak{b}$ is nilpotent of order $1$.

\subsection{Nilpotent Lie algebras - Uniform conditions}

A trivial condition in order to ensure that the Lie Bracket $[ B_{(i-1,n-i)} , B_{(n-i,i-1)} ]$ is zero is to set $U_i =0$ for $i=1,\dots ,n$. This condition coupled with $p_{0,n}=0$ annihilating the bracket $[B_{(n,-1)} , B_{(-1,n)}]$ is in fact more powerful. Indeed, we have:

\begin{lem}
\label{structure1}
Let $X$ be a polynomial vector field of the form (\ref{form}) where $P$ is homogeneous of degree $d$. If the coefficients of $P$ satisfy 
\begin{equation}
p_{d-i+1,i-1} -\bar{p}_{i,d-i} =0 ,\ \ i=1, \dots ,d ,\ \mbox{\rm and}\ \ p_{0,d} =0,
\end{equation}
then the Lie algebra $\mathfrak{b}$ is nilpotent of order $1$.
\end{lem}

\begin{proof}
The condition $p_{0,d}=0$ implies not only that the bracket $[B_{(n,-1)} , B_{(-1,n)}]$ is zero, but also that $B_{(n,-1)} =B_{(-1,n)} =0$. Then, the Lie algebra $\mathfrak{b}$ is only generated by the operators $B_{(i-1,d-i)}$ for $=1,\dots ,d$. Again the condition $p_{d-i+1,i-1} -\bar{p}_{i,d-i} =0$, $i=1, \dots ,d$ implies that all the operators $B_{(i-1,d-i)}$ are of the form  $B_{(i-1,d-i)} = p_{i,d-i} x^{i-1} y^{d-i} X_E$, where $X_E$ is the classical {\it Euler vector field} defined by 
$X_E = x\partial_x +y\partial_y$. The Euler vector field acts trivially on each monomials. Precisely, we have $X_E (x^{i-1} y^{d-i} )=(d-1)\, x^{i-1} y^{d-i}$. The Lie bracket of $B_{(i-1,d-i)}$ with $B_{(j-1,d-j)}$ is then easily computed. We obtain:
\begin{equation}
\left .
\begin{array}{lll}
[B_{(i-1,d-i)} ,B_{(j-1,d-j)} ] & = & p_{i,d-i} p_{j,d-j} \left [ 
x^{i-1}y^{d-i} \left ( X_E [x^{j}y^{d-j}] \partial_x
+X_E [x^{j-1}y^{d-j+1}] \partial_y  \right ) \right . \\
 & & - \left . x^{j-1}y^{d-j} \left ( X_E [ x^{i}y^{d-i}]\partial_x +
+X_E [x^{i-1}y^{d-i+1}] \partial_y  \right ) 
\right ] ,\\
 & = & 0 ,
\end{array}
\right .
\end{equation}
which implies that $\mathcal{C}^2 (\mathfrak{b} )=0$. This concludes the proof.
\end{proof}

\subsection{Resonant subset - Holomorphic conditions}

We have pointed out the special role of the coefficients $p_{i,n-i}$ in the computations. We introduce the resonant subset $\mathfrak{b}_{res}$ of $\mathfrak{b}$ generated by all the Lie bracket of $[B_{\mathbf{n}}]$ such that $\omega (\mathbf{n} )=0$, $\mathbf{n} \in A(X)^*$. As in the previous Section, we prove the following result:

\begin{lem}
\label{holom}
Let $X$ be a polynomial vector fields of the form \ref{form}. We assume that the coefficients of $P$ satisfies 
\begin{equation}
p_{i,n-i}=0,\  n=2,\dots ,d,\  i=0,\dots ,d-1 .
\end{equation}
then the subset $\mathfrak{b}_{res}$ of the Lie algebra $\mathfrak{b}$ is trivial, i.e. $\mathfrak{b}_{res} =\{ 0\}$.
\end{lem}

\begin{proof}
As usual, the previous condition implies that $p_{0,n}=0$ for all $n=2,\dots ,d$ so that $B_{(n,-1)}=B_{(-1,n)}=0$. We then concentrate on the operators $B_{(i-1,n-i)}$ for $i=1,\dots , n$ and $n=2,\dots ,d$. The assumption implies that $B_{(i-1,n-i)}=0$ for $i=2,\dots ,n-1$ and 
\begin{equation}
B_{(0,n-1)} = y^n \bar{p}_{n,0} \partial_y ,\ \ \mbox{\rm and}\ \ B_{(n-1,0)} = x^n p_{n,0} \partial_x .
\end{equation}
for $n=2,\dots ,d$. A simple computation shows that for all $n,m \in {2,\dots ,d}$, we have
\begin{equation}
[ B_{(0,n-1)} ,B_{(m-1 ,0)}] =0 .
\end{equation}
Moreover, we have $[ B_{(0,n-1)} ,B_{(0,m-1)}]= y^{n+m-1} c_{n,m} \partial_y$, where $c_{n,m}$ is a constant. As a consequence, we have
$[B_{(0,n_1 -1),\dots ,(0,n_k -1)} ]=y^{n_1 +\dots +n_k -1} c_{n_1 ,\dots ,n_k} \partial_y$. The same is true for $[ B_{(m_1 -1 ,0) ,\dots ,(m_r -1 ,0)}]$ replacing $y$ by $x$.\\

A resonant word must mix letters with positive weight and negative weight. Assume that a given word begins with some $(0,n_1 -1),\dots ,(0,n_k -1)$ for a given $k\in \N^*$. One must have at least one operator of the form $B_{(m_1 -1 ,0)}$ in order to have a resonant sequence. As $[B_{(0,n_1 -1),\dots ,(0,n_k -1)} ]$ depends only on $y \partial_y $ and $B_{(m_1 -1 ,0)}$ of $x\partial_x$, the corresponding Lie bracket is zero. Then all $[B_{\mathbf{n}} ]$ such that $\omega (\mathbf{n} )=0$ is zero. This concludes the proof.
\end{proof}

\section{Uniform isochronism}

Following R. Conti \cite{conti}, an isochronous center of a vector field $X$ is said to be {\it uniform} if all the periodic orbits have the same period. This condition is equivalent to (see \cite{conti},$\S$.19, Definition 19.1 p.28) the following relation on the polynomial $P$:
\begin{equation}\tag{UI}
yP(x,y)=x\overline{P(x,y )} .
\end{equation}
This conditions induces specific relations on the coefficients of the polynomial:

\begin{lem}
A polynomial $P$ of degree $d$ satisfies the uniform isochronicity condition (UI) if and only if the coefficients $p_{i,j}$, $i+j=n$, $n=2,\dots ,d$ satisfy 
\begin{equation}
\label{cui}
p_{0,n}=0 \ \mbox{\rm and}\ \ 
p_{i,n-i} =\overline{p}_{n-i+1 ,i-1} ,\ n=2,\dots ,d,\ i=1,\dots ,n .
\end{equation}
\end{lem}

Using this Lemma and the structure Lemma \ref{structure1} for the Lie algebra generated by the comoulds in $\mathbf{B}(X)$, we deduce the following Lemma: 

\begin{lem}
Let $X$ be a polynomial vector field of the form (\ref{form}) where $P$ is homogeneous of degree $d$. Assume that the coefficients of $P$ satisfy 
\begin{equation}
p_{0,d}=0, \ \ 
p_{i,d-i} =\overline{p}_{d-i+1 ,i-1} ,\ i=1,\dots ,d .
\end{equation}
Moreover, if $d$ is odd with $d=2m+1$, we suppose that $p_{m+1,m} =0$. Then, the vector field $X$ is isochronous.
\end{lem}

\begin{proof}
By Lemma \ref{structure1}, the first conditions imply that the Lie algebra $\mathfrak{b}$ is nilpotent of order $1$. We deduce from Lemma \ref{fond3} that any prenormal form is reduced to  
\begin{equation}
Pran (X) = X_{lin} + \di\sum_{n\in A(X),\ \omega (n)=0} Pran^n B_n .
\end{equation}
If $d$ is even, the set of resonant letters is empty and $Pran(X)=X_{lin}$. If $d$ is odd, then we have only one resonant letter in $\mathbf{B} (X)$ given by $n_m =(m+1 ,m)$. Then a prenormal form is given by 
\begin{equation}
Pran (X) = X_{lin} + Pran^{n_m} B_{n_m} .
\end{equation}
As $p_{m+1,m}=0$ we obtain $B_{n_m} =0$ and $Pran(X)=X_{lin}$. This concludes the proof.
\end{proof}

\section{Holomorphic isochronous centers}

Following Conti \cite{conti}, a vector field is said to satisfy the {\it Cauchy-Riemann}\footnote{In \cite{conti}, this condition is stated for the underlying real vector field.} conditions if 
\begin{equation}\tag{CR}
\partial_y P =0 . 
\end{equation}

The Cauchy-Riemann conditions impose strong constraints on the coefficients of $P$: 

\begin{lem}
A polynomial $P$ of degree $d$ satisfies the Cauchy-Riemann conditions (CR) if and only if $p_{i,n-i} =0$ for $i=0,\dots ,n-1$, $n=2,\dots ,d$.
\end{lem}

We deduce:

\begin{lem}
Let $X$ be a polynomial vector fields of degree $d$ of the form (\ref{form}). Assume that $P$ satisfies the Cauchy-Riemann condition. Then the vector field $X$ is linearisable.
\end{lem}

\begin{proof}
Formula (\ref{projprenor}) implies that any prenormal form associated to a mould $Pran^{\bullet}$ can be written as 
\begin{equation}
\label{formu}
Pran(X)=X_{lin} +\di\sum_{r\geq 1} \di\frac{1}{r}\di\sum_{\mathbf{n} \in A^*(X) ,\, l(\mathbf{n} )=r,\, \omega (\mathbf{n} )=0} Pran^{\bullet} [B_{\mathbf{n}}] .
\end{equation}
By definition all the Lie brackets in (\ref{formu}) belong to $\mathfrak{b}_{res}$. Using Lemma \ref{holom}, the Cauchy-Riemann conditions imply that $\mathfrak{b}_{res}$ is trivial. As a consequence, all the Lie brackets $[B_{\mathbf{n}}]$ reduce to zero for $\mathbf{n}\in A^*(X)$ such that $\omega (\mathbf{n} )=0$. We conclude that any prenormal form reduces to the linear one and the vector field $X$ is formally linearisable.
\end{proof}

The conditions on the coefficients correspond to the characterizations of holomorphic isochronous centers.

\section{Linearisability and complexity}

Following J. Ecalle and D. Schlomiuk in \cite{es}, we introduce the following problem :\\

\ni {\bf Minimal complexity of the linearisability problem}: {\it Let $d\in \N^*$, $d\geq 2$ and $X$ a polynomial vector field of degree $d$ given by (\ref{form}). We denote by $\mbox{\rm lin} (d)$ the minimal number of algebraic relations depending on the coefficients of $P$ which induce analytic linearisability. Can we determine a bound or an explicit formula for $\mbox{\rm lin}(d)$ ?} \\

Using our approach, we understand that this number depends on the complexity of the isochronous center conditions. In particular, each condition $C_i$ where $i=1,\dots ,m$ of an isochronous center are determined by a finite family of polynomial $C_i (p)$ of degree $c_i$. A natural notion of {\it complexity} is given by:

\begin{dfn}
Let $S\in \C^m$ be an algebraic set. Consider a representation $(R)$ of $S$. The complexity of the representation $(R)$ is defined by the triplet $(m,P(R),C(R))$, composed by the following data:
\begin{enumerate}
\item The dimension of the ambient space $\C^m$,
\item The number of condition in $(R)$,
\item The maximal degree of polynomials defining the conditions in $(R)$.
\end{enumerate}
\end{dfn} 

In our case, the dimension of the ambient space is fixed by $d$ and is given by the number $m(d)$ of coefficients of a generic polynomial of degree $d$. We then introduce the notion of {\it geometric complexity} for an isochronous condition:

\begin{dfn}
An isochronous condition $C$ is said of geometric complexity $(q,m)$ if it admits a representation made of $q$ polynomial identities of degree at most $m$.
\end{dfn}

The previous notion can be used to refine the question raised by J. Ecalle and D. Schlomiuk:\\

\ni {\bf Minimal geometric complexity of the linearisability problem}: {\it Let $d\in \N^*$, $d\geq 2$ and $X$ a polynomial vector field of degree $d$ given by (\ref{form}). We denote by $\mbox{\rm glin} (d)$ the minimal geometric complexity of an isochronous center condition. Can we find a bound or a formula for $\mbox{\rm glin}(d)$ ?} \\

As expected by J. Ecalle and D. Schlomiuk in \cite{es}, this question is more tractable than the initial one. In particular, our two examples already give conditions for which the minimal degree is attained by an isochronous center condition. As a consequence, we can look over the set of center condition to compute in each case the couple $g(C)$. We have
\begin{equation}
g_{Holo} (d)=(d,1),\ \ g_{Uni}(d)=(d+1,1)\ \mbox{\rm if}\ d\ \mbox{\rm is even and}\ g_{Uni}(d)=(d+2,1)\ \ 
\mbox{\rm if}\ d\ \mbox{\rm is odd.}
\end{equation}
Conditions ensuring the nilpotent character of $\mathfrak{b}$ or the triviality of $\mathfrak{b}_{res}$ are always of degree $1$ as they can be read on the Lie bracket of homogeneous vector fields $B_n$, $n\in A(X)$. Other isochronous center conditions depend on the interplay between moulds and comoulds and generate polynomial identities of degree at least $2$. As a consequence, we are leaded to the following conjecture:\\

\ni {\bf Conjecture}: {\it The number $glin(d)$ is equal to $(d,1)$ corresponding to holomorphic isochronous center.}

\end{document}